\documentclass[11pt,reqno,oneside]{amsart}
\usepackage[english]{babel}
\usepackage[utf8]{inputenc}
\usepackage[mathscr]{eucal}
\usepackage{color}
\usepackage{xcolor}
\usepackage{amsfonts}   
\usepackage{amssymb}
\usepackage{latexsym}
\usepackage{enumerate}
\usepackage[colorlinks,citecolor=purple,linkcolor=blue]{hyperref}

\usepackage{geometry}

\let\nc\newcommand

\nc{\la}{\label}

\newtheorem{theorem}{Theorem}[section]
\newtheorem{corollary}[theorem]{Corollary}
\newtheorem{lemma}[theorem]{Lemma}
\newtheorem{proposition}[theorem]{Proposition}
\theoremstyle{definition}
\newtheorem{definition}[theorem]{Definition}
\newtheorem{example}[theorem]{Example}
\newtheorem{notation}[theorem]{Notation}
\theoremstyle{remark}
\newtheorem{remark}[theorem]{Remark}

\newcommand{\Ann}{{\rm{Ann}}}

\newcommand{\Ker}{{\rm{ker}}}

\newcommand{\im}{{\rm{im}}}

\newcommand{\into}{\,\,\hookrightarrow\,\,}

\newcommand{\gr}{{\rm gr}}
\newcommand{\GK}{{\rm GK}}

\newcommand{\AMod}{A\text{-}\mathsf{Mod}}
\newcommand{\Amod}{A\text{-}\mathsf{mod}}

\newcommand{\End}{\operatorname{End}}

\begin{document}

\title{Hilbert-Samuel Polynomials 
for Algebras\\with Special Filtrations}
\author{Jonas T. Hartwig}
\address{Department of Mathematics, Iowa State University, Ames IA 50011, USA}
\email{jth@iastate.edu}
\urladdr{http://jthartwig.net}
\thanks{J.T.H. is partially supported by the Army Research Office grant W911NF-24-1-0058.}

\author{Erich C. Jauch}
\address{Department of Mathematics \& Physics, Westminster College (Missouri), Fulton MO 65251}
\email{erich.jauch@westminster-mo.edu}
\urladdr{http://ecjauch.com}

\author{Jo\~ao Schwarz}
\address{Mathematicis - College of Sciences, SUSTech, 1088 Xueyuan Avenue, Shenzhen 518055,
P.R. China}
\email{jfschwarz.0791@gmail.com}

\subjclass[2020]{Primary: 16P90 13H15 13D40 Secondary: 16W70 14F10}
\keywords{Gelfand-Kirillov dimension, Hilbert-Samuel polynomial, multiplicity, holonomic module}

\begin{abstract}
    The notion of multiplicity of a module first arose as consequence of Hilbert's work on commutative algebra, relating the dimension of rings with the degree of certain polynomials. For noncommutative rings, the notion of multiplicity first appeared in the context of modules for the Weyl algebra in Bernstein's solution of the problem of analytic continuation posed by I. Gelfand. The notion was shown to be useful to many more noncommutative rings, especially enveloping algebras, rings of differential operators, and quantum groups. In all these cases, the existence of multiplicity is related to the existence of Hilbert-Samuel polynomials. In this work we give an axiomatic definition of algebras with a notion of multiplicity, which we call \emph{very nice} and \emph{modest algebras}. We show, in an abstract setting, how the existence of Hilbert-Samuel polynomials implies the existence of a notion of multiplicity. We apply our results for the category of \emph{min-holonomic} modules --- a notion which coincides with holonomic modules for simple algebras --- and that shares many similarities with it. In particular, we generalize the usual results in the literature that are stated for Ore domains, in the more general context of prime algebras, and we show that rational Cherednik algebras admit a notion of multiplicity.

\end{abstract}

\maketitle

\tableofcontents

\section{Introduction}

The Gelfand-Kirillov dimension of an affine algebra $A$ over a base field $k$\footnote{in other words, a finitely generated associative unital not-necessarily-commutative $k$-algebra}, denoted in this paper by $\GK \, A$ or $\GK(A)$, is a type of measure of growth of the algebra. It was introduced in \cite{Gelfand}, along with the Gelfand-Kirillov transcendence degree. That paper posed the famous Gelfand-Kirillov Conjecture: if $\mathfrak{g}$ is an algebraic Lie algebra over a field of characteristic $0$, the skew-field\footnote{\textbf{}division ring} of fractions of its enveloping algebra is a suitable Weyl field. (We now know that the Conjecture is false \cite{{AOvdB}}.) Let $W_n(k)$ denote the rank $n$ Weyl algebra over a field $k$ of zero characteristic. The first use of the Gelfand-Kirillov dimension was to show that if $W_n(k) \simeq W_m(k)$, then $m=n$.
The notion of growth of algebras is also clearly related to the notion of growth of other algebraic structures \cite{Zelmanov}, particularly groups \cite[Chapter 11]{KL}.

In 1976, Borho and Kraft \cite{BK} developed the properties of the Gelfand-Kirillov dimension, and soon its relevance in ring theory became apparent. In particular, the dimension became an important tool in the study of enveloping algebras (\cite{Jantzen}, \cite{French}). With the development of the theory of quantum groups, again the Gelfand-Kirillov dimension proved to be very useful (\cite{Brown}). In this last case, powerful algorithms exists to compute this invariant (see \cite{BOOK}).

It also became very relevant to noncommutative projective geometry, where it is a good substitute of the Krull dimension (cf. \cite{SvdB}, \cite{Rogalski}). Indeed, if $A$ is an affine commutative algebra, $\GK(A)=\operatorname{Krull}(A)$, so the Gelfand-Kirillov dimension can be seen as a generalization of the Krull dimension.

At the International Congress of Mathematicians in Amsterdam, 1954, I. M. Gelfand posed an important problem of analytic continuation of certain meromorphic functions. The problem was solved, but its solution involved some heavy machinery, such as Hironaka's resolution of singularities. Later, I. N. Bernstein \cite{INBernstein} offered an elementary proof, using just basic considerations of the representation theory of the Weyl algebra; namely, he introduced the notion of \emph{holonomic modules}, which are very important in algebra and geometric analysis \cite{Borel} \cite{Hotta}. For a good account of this story, see \cite[Chapter 8]{KL}, where Bernstein's solution to Gelfand's problem is shown in detail.

He also introduced a notion of multiplicity, originally from commutative algebra \cite{Matsumura}, for the first time for noncommutative rings. The notion of holonomic modules was generalized for $D$-modules over more general algebraic varieties and other classes of algebras, such as enveloping algebras; and the same with the notion of multiplicity. All notions of multiplicity of an algebra appearing in the literature \cite[Chapter 12]{KL} share a common core. We define axiomatically algebras that satisfy these properties: we introduce the concept of \emph{very nice} and \emph{modest algebras} (Definition \ref{very-nice}).

In all known cases, such as differential operators, enveloping algebras or quantum groups, the existence of multiplicity follows from the fact that these algebras admit nice filtrations with Hilbert-Samuel polynomials. We axiomatize this in Theorem \ref{theorem-Hilbert-Samuel} and prove that the rational Cherednik algebras admits a notion of multiplicity --- a fact that seems to have escaped the attention of specialists (\cite{Thompson}, \cite{Losev}, \cite{Bellamy}). For very nice and modest algebras we introduce a small variation of the notion of holonomicity: the \emph{min-holonomic} modules, following the approach of \cite{BF}. Our notion coincides with the usual one in the case the algebra is simple, and the present authors do not know any examples of algebras for which the two notions differ. For rings of differential operators (\cite{Bjork}, \cite{Coutinho}) and enveloping algebras (\cite{French}, \cite{Jantzen}) the combination of the notions of Gelfand-Kirillov dimension and multiplicity proved to be very fruitful. To the best of our knowledge, all uses of multiplicity in the literature were for Ore domains, such as differential operators, enveloping algebras or quantum groups. With only small changes, we obtained analogues of results of the category of holonomic modules for Ore domains
for the category of min-holonomic modules in case the algebra is only assumed to be a prime ring. We believe these results to be new. As consequence, for instance, our theory can be applied to category of holonomic modules for rational Cherednik algebras when the parameter is regular (and so the algebra is simple and holonomic and min-holonomic modules coincide).

Finally, we survey the three most important classes of filtered algebras which have Hilbert-Samuel polynomials: namely, almost commutative algebras, somewhat commutative algebras and filtered semi-commutative algebras. We also consider an obvious (but as far as we know, new) generalization of the notion of somewhat commutative algebras, which we call \emph{vaguely commutative algebras} (cf. Definition \ref{vaguely-commutative}). We show that spherical subalgebras of any symplectic reflection algebra \cite{EG} are modest algebras, and hence always have a notion of multiplicity --- another fact not mentioned by specialists. We also consider the Poincaré series of filtered algebras, for it can also give information about the existence or not of Hilbert-Samuel polynomials; and filtrations by ordered semigroups other then $\mathbb{N}$, where it has been shown that one can essentially reduce multifiltrations to $\mathbb{N}$-filtrations.

The main point of the paper is to lay down a foundation for the study algebras with a notion of multiplicity, which is as useful for representation theory as it was for the case of the Weyl algebra.
Our main references are \cite{KL}, \cite[Chapter 8]{McConnell}, \cite{Lorenz}, \cite{BF} and \cite{McConnnel2}.

The structure of the paper is as follows. In Section \ref{sec:preliminaries}, we recall some basic facts about Gelfand-Kirillov dimension and filtered algebras. The notions of very nice and modest algebras are introduced in Section \ref{sec:verynice}.
In Section \ref{sec:algebras-with-special-filtrations}, we consider several important classes of algebras equipped with a filtration subject to some conditions, their interrelationships and properties. In particular we introduce the class of vaguely commutative algebras.
Lastly, in Section \ref{sec:more-general-filtrations}, we study semigroup filtrations techniques of re-filtering.

\section{Preliminaries}\label{sec:preliminaries}

\subsection{The basics of the Gelfand-Kirillov dimension}

We follow \cite[Chapters 1, 2]{KL} and \cite[Section 8.1]{McConnell} closely. Throughout, $k$ denotes the base field, and is arbitrary unless otherwise explicitly stated.

\begin{definition}
Let $\Phi$ be the set of functions $f: \mathbb{N} \rightarrow \mathbb{R}$ such that there exists an $n_0 \in \mathbb{N}$ with:
\[f(n)>0\text{ and } f(n+1) \geq f(n), \quad\forall \, n \geq n_0.\]
For two functions in $\Phi$, we write $f \leq^* g$ if there are constants $c, m \in \mathbb{N}$ such that $f(n) \leq cg(mn)$. If $g \leq^* f$ also, we write $f \sim g$. We represent the class of $f$ in $\Phi/\sim$ by $\mathcal{G}(f)$. This equivalence class of $f$ is called its \emph{growth}, and the relation induced by $ \leq^* $ in $\Phi/\sim$ is denoted $\leq$.
\end{definition}
\begin{notation}
Let $y \in \mathbb{R}, y \geq 0$. The growth of the function $n \mapsto n^y$ is denoted $\mathcal{P}_y$.
Let $y \in \mathbb{R}, y>0$. The growth of the function $n \mapsto e^{n^y}$ is $\mathcal{E}_y$.
\end{notation}

If $f(n)=\ln(n+1)$, $\mathcal{G}(f) > \mathcal{P}_0$ but $\mathcal{G}(f) < \mathcal{P}_\varepsilon$ , for any $\varepsilon >0$. So there is no analogue of the archimedean property for this order.

Finally, we remark that not every pair $\mathcal{G}(f)$, $\mathcal{G}(h)$ is comparable.

\begin{notation}\label{asymptotic}
For $f \in \Phi$,
\[\gamma(f)= \limsup_{n \to \infty}\big(\log_n \, f(n)\big).\]
\end{notation}

\begin{proposition}\label{growth}
    Let $f,g \in \Phi$.
    \begin{enumerate}[{\rm (i)}]
        \item If $\mathcal{G}(f)=\mathcal{G}(g)$, then $\gamma(f)=\gamma(g)$.
        \item $\gamma(f+g)=\sup \{ \gamma(f), \gamma(g) \}$.
        \item If $f(n)=p(n)$ for $n\gg 0$ and $p(x) \in \mathbb{R}[x]$, then $\gamma(f)=\deg \, p$.
        \item If $\gamma(f)=\lim_{n \to \infty} \log_n \, f(n)$, $\gamma(fg)= \gamma(f)+\gamma(g)$.
        \item If $g(n) \leq f(an+b), \, n\gg 0$, where $a,b \in \mathbb{N}$, then $\gamma(g) \leq \gamma(f)$.
    \end{enumerate}
\end{proposition}

\begin{proof}
    \cite[Lemma 2.1b)]{KL}, \cite[Lemma 8.1.7]{McConnell}
\end{proof}

\begin{definition}
Let $A$ be an affine algebra and $V$ a finite dimensional generating space for $A$. Let $d_V(n)=\dim \, \sum_{i=0}^n V^i$. Then $\mathcal{G}(d_V)$ is independent of the choice of the generating space $V$ (\cite[Lemma 1.1]{KL}). Hence,
\[ \gamma(d_V)=\limsup_{n \rightarrow \infty} \big(\log_n(d_V(n))\big),\]
is also independent of $V$, so by Proposition \ref{growth}(i) this number is well-defined, and is called the \emph{Gelfand-Kirillov dimension} of $A$, denoted $\GK \, A$. We also put $\mathcal{G}(A)=\mathcal{G}(d_V)$, where $V$ is any choice of frame for $A$.
\end{definition}

\begin{remark}
    This definition makes sense and is useful for non-associative algebras, like Lie algebras (\cite[Chater 12, Section 1]{KL}
\end{remark}

\begin{remark}
When $1 \in V$, $d_V(n)$ reduces to $\dim \, V^n$, and $V$ is called a frame.
\end{remark}

\begin{definition}
    Let $A$ be an affine algebra. We say $A$ has 
    \begin{itemize}
        \item \emph{polynomial growth} if $\mathcal{G}(A)=\mathcal{P}_m$ for some natural $m$,
        \item \emph{exponential growth} if $\mathcal{G}(A)=\mathcal{E}_1$,
        \item \emph{subexponenial growth} or \emph{intermediate growth} if $\mathcal{G}(A)< \mathcal{E}_1$ but $\mathcal{G}(A) \nleq \mathcal{P}_m, \, \forall \, m \in \mathbb{N}$.
    \end{itemize}
\end{definition}

\begin{example}
The Gelfand-Kirillov dimension of an affine algebra is $0$ if and only if it is finite-dimensional. The free algebra $k \langle x, y \rangle$ has exponential growth, and hence its Gelfand-Kirillov dimension is $\infty$. Indeed, $V=kx \oplus ky$ is a generating space and $\dim \, \sum_{i=0}^n \dim \, V^i=1+2+ \ldots 2^n= 2^{n+1}-1$.
\end{example}

The following Proposition is canonical and very useful:

\begin{proposition}\label{convenient-proposition}\phantom{X}
\begin{enumerate}[{\rm (a)}]
\item If $0 \neq f \in \mathbb{Q}[x]$ is polynomial of degree $d$, then there are rational numbers $a_0, a_1, \ldots, a_d$, such that $f(n) = a_d \binom{n}{d} + a_{d-1} \binom{n}{d-1}+ \ldots a_1 \binom{n}{1}+a_0$, $\forall \, n \in \mathbb{N}$.
\item The following properties of a function $f: \mathbb{N} \rightarrow \mathbb{Q}$ are equivalent:
\begin{enumerate}[{\rm (i)}]
\item There exists $m \in \mathbb{N}$ such that for all $n \geq m$ we have $f(n) = a_d \binom{n}{d} + a_{d-1} \binom{n}{d-1}+ \ldots a_1 \binom{n}{1}+a_0$
\item There exists $a_1, \ldots, a_d \in\mathbb{Q}$ and $m \in \mathbb{N}$ such that for all $n \geq m$,
\[ f(n+1)-f(n)=a_d \binom{n}{d-1} + \ldots a_2 \binom{n}{1} + a_1.\]
\end{enumerate}
\item If $f(n)$ is expressed as in \emph{(a)} and $f(n) \in \mathbb{Z}, n\gg 0$, $a_0, a_1, \ldots, a_d \in \mathbb{Z}$
\item If $f(n)$ is expressed as in \emph{(a)}, and for $n\gg 0$, $f(n) \in \mathbb{N}$ and $f(n+1)-f(n) \geq 0$, $a_d$ is a positive integer, called the \emph{Bernstein number} of $f$.
\end{enumerate}
\end{proposition}

\begin{example}
    Let $A=k[x_1, \ldots, x_d]$, $V=kx_1 \oplus \ldots kx_n$. Then

    \[ \dim(V^{n+1}) = \binom{n+1+d-1}{d-1}=\binom{n+1}{d-1} \]

is a polynomial in $d$ with degree $d-1$. Since $\dim(V^{n+1})=d_V(n+1)-d_V(n)$ it follows by item b) of the previous Proposition that $d_V$ is a polynomial in $n$ with degree $d$. Hence $GK \, k[x_1,\ldots,x_n]=d$.

\end{example}

Note that in this example, the Gelfand-Kirillov dimension and the Krull dimension give the same number. This is no coincidence:

\begin{theorem}
    Let $A$ be an affine commutative algebra. Then $GK(A)=\operatorname{Krull} \, A$.
\end{theorem}
\begin{proof}
    \cite[Theorem 4.5]{KL}.
\end{proof}

\begin{example}\label{Smith}
    Let $\mathfrak{L}$ be the infinite-dimensional Lie algebra with basis $\{ x, y_1, y_2, \ldots \}$ and brackets given by $[x, y_i]=y_{i+1}, \, [y_i, y_j]=0$. Let $A=U(\mathfrak{L})$. Then $\mathcal{G}(A)=\mathcal{G}\big(\exp(\sqrt{n})\big)$ has subexponential growth \cite[Example 1.3]{KL}.
\end{example}

\begin{proposition}
    If $A$ is infinite-dimensional affine algebra, then $\mathcal{P}_1 \leq \mathcal{G}(A) \leq \mathcal{E}_1$. In particular, if $A$ is infinite-dimensional, its Gelfand-Kirillov dimension is at at least $1$.
\end{proposition}

\begin{proof}
    \cite[Proposition 1.4]{KL}.
\end{proof}

\begin{theorem}
    There is no infinite-dimensional affine algebra with Gelfand-Kirillov dimension in the open real interval $(1,2)$. This is known as \emph{Bergman's Gap}. If $\alpha$ is any real number $\geq 2$, there exists an infinite-dimensional algebra $A$ with $\GK \, A = \alpha$.
\end{theorem}
\begin{proof}
    \cite[Chapters 1, 2]{KL}.
\end{proof}

\begin{definition}
    Let $M$ be a finitely generated left $A$-module, where $A$ is an affine algebra. Choose a finite-dimensional space $F \subset M$ that generates it as a module, and let $V$ be a frame for $A$. Let $d_{F,V}(n)=\dim \, V^n F$. The \emph{Gelfand-Kirillov dimension} of $M$ is
    \[ \GK \, M=\GK_A\, M=\limsup_{n \to \infty} \, \log_n \, d_{F,V}(n), \]
    and the value is independent of the choices of $F, V$ by \cite[Proof of Lemma 1.1]{KL}.
\end{definition}

\subsection{Filtered algebras}

For the sake of simplicity, we will consider only $\mathbb{N}$-gradings and filtrations. 

\begin{definition}
    A \emph{grading} on an algebra $A$ is a collection of subspaces $\{A_i\}_{i=0}^\infty$ such that 
    $A=\bigoplus_{i=0}^\infty A_i $ and $A_i A_j \subset A_{i+j}$. Note that necessarily $k \subset A_0$.
    A grading is \emph{finite-dimensional grading} if each $A_i$ is a finite-dimensional vector space.
    A \emph{graded algebra} is an algebra together with a grading. A graded algebra is \emph{connected} if its grading is finite-dimensional and $A_0=k$.
\end{definition}

\begin{definition}
    Let $A$ be a graded algebra. A \emph{grading} on an $A$-module $M$ is a collection of subspaces $\{M_i\}_{i=0}^\infty$ such that $M=\bigoplus_{i=0}^\infty M_i$ and $A_i M_j \subset M_{i+j}$. If each $M_i$ is finite-dimensional, we say that the grading is \emph{finite-dimensional}. A \emph{graded $A$-module} is an $A$-module $M$ together with a grading on $M$.
\end{definition}

\begin{definition}
    A \emph{filtration} on an algebra $A$ is 
    an increasing chain of subspaces $\mathcal{F}=\{ F_i \}_{i=0}^\infty$,  $i\le j \Rightarrow F_i \subset F_j$, such that $1 \in F_0$, $F_i F_j \subset F_{i+j}$ and $A= \bigcup_{i=0}^\infty F_i$. The filtration is \emph{finite-dimensional} if each $F_i$ is finite-dimensional.
    A \emph{filtered algebra} is an algebra together with a filtration.
    We denote by $\gr_\mathcal{F} \, A$ the \emph{associated graded algebra}. As a vector space it is
    \[ \gr_\mathcal{F} \, A = \bigoplus_{i=0}^\infty \big(A_i/A_{i-1}\big),\]
    with the convention that $A_{-1}=0$. The multiplication in $\gr_\mathcal{F}\, A$ is given by the formula $(x+ A_{i-1})(y+A_{j-1})=(xy+A_{i+j-1})$.
\end{definition}

\begin{definition}
    Let $A$ be a filtered algebra. A \emph{filtration} on an $A$-module $M$ is an increasing chain of subspaces $\Omega=\{ M_i \}_{i \geq 0}$, $i\le j\Rightarrow M_i\subset M_j$, such that $A_iM_j\subset M_{i+j}$ and $M=\bigcup_{i=0}^\infty M_i$. If each $M_i$ is finite-dimensional, the filtration $\Omega$ is \emph{finite-dimensional filtration}. We denote by $\gr_\Omega \, M$ the \emph{associated graded module}, which is a module over $\gr_\mathcal{F} \, A$. As a vector space,
    \[ \gr_\Omega \, M = \bigoplus_{i=0}^\infty \big(M_i/M_{i-1}\big), \]
    with the convention that $M_{-1}=0$. The action of $\gr_\mathcal{F}\,A$ on $\gr_\Omega\, M$ is given by $(a+ A_{i-1})(m+M_{j-1})=(am +M_{i+j-1})$.
\end{definition}

The following techniques are common in ring theory.

\begin{proposition}\label{filtered-graded}
Let $A$ be an algebra with filtration $\mathcal{F}$.
\begin{enumerate}[{\rm (i)}]
    \item If  $\gr_\mathcal{F} \, A$ is an affine algebra, so is $A$.
    \item If  $\gr_\mathcal{F} \, A$ is a Noetherian algebra, so is $A$.
    \item If  $\gr_\mathcal{F} \, A$ is prime ring, so is $A$.
    \item If  $\gr_\mathcal{F} \, A$ is a domain, so is $A$.
\end{enumerate}
\end{proposition}

\begin{proof}
    \cite{KL}, \cite{McConnell}.
\end{proof}

As we know, not every domain is an Ore domain. For instance, the free associative algebra $k \langle x_1, \ldots, x_n \rangle$ is not an Ore domain (when $n>1$). In fact, as shown by Malcev (\cite[Theorem 9.11]{Lam2}), there are noncommutative domains that cannot even be embedded into division rings.
So it will be convenient to have at hand general results providing sufficient conditions for a domain to be an Ore domain.

\begin{proposition}\label{it-is-Ore}
 A domain $R$ is an Ore domain if any of the following conditions hold:
\begin{enumerate}[{\rm (i)}]
    \item $R$ is Noetherian,
    \item $R$ does not contain the free algebra on two elements,
    \item $R$ is a PI-algebra,
    \item $R$ is of subexponential growth; in particular, if $\GK \, R<\infty$.
\end{enumerate}
\end{proposition}

\begin{proof}
    That (i) is sufficient for $R$ to be an Ore domain follows from \cite[2.1.15]{McConnell}.
    By Jategaonkar's Theorem \cite[Proposition 4.13]{KL}, if a domain does not contain the free algebra in two elements as a subalgebra, it is an Ore domain. This proves that $R$ is an Ore domain if (ii) holds. A subalgebra of a PI-algebra is also a PI-algebra; and the free associative algebras does not satisfy any polynomial identity. So (iii) implies (ii). Finally, if $R$ has subexponential growth, then it cannot contain the free algebra on two elements. Thus (iv) also implies (ii).
\end{proof}

\begin{example}
It is well known that if $\mathfrak{g}$ if a finite dimensional Lie algebra, then $U(\mathfrak{g})$ is an Ore domain. However, this can happen even if $\mathfrak{g}$ is infinite dimensional. If $\mathfrak{L}$ is the infinite dimensional Lie algebra of Example \ref{Smith}, we saw $U(\mathfrak{L})$ has subexponential growth.
\end{example}

\begin{proposition}\label{Proposition 6.6}
    Let $(A,\mathcal{F})$ be a filtered algebra and $(M,\Omega)$ a filtered $A$-module. Then:
    \begin{enumerate}[{\rm (i)}]
    \item $\GK_{\gr_\mathcal{F} \, A}(\gr_\Omega \, M) \leq \GK_A(M)$.
    \item If $\mathcal{F}$ and $\Omega$ are finite-dimensional filtrations, then $\gr_\mathcal{F} \, A$ is affine and $\gr_\Omega \, M$ is a finitely generated $\gr_\mathcal{F} \, A$-module. Moreover, setting $d_M(n)=\operatorname{dim} M_n$,
    \[ \GK_{\gr_\mathcal{F} \, A}(\gr_\Omega \, M) = \GK_A(M)= \gamma(d_M).\]
    \end{enumerate}
\end{proposition}
\begin{proof}
    \cite[Lemma 6.5, Proposition 6.6]{KL}.
\end{proof}

\begin{example}
    Let $\mathfrak{g}$ be a finite dimensional Lie algebra. Call $U$ its enveloping algebra. Define a filtration $\mathcal{F}=\{ U_i \}_{i=0}^\infty$, with $U_0=0$, $U_n=(k \oplus \mathfrak{g})^n$ , $n>0$. Then, the PBW-Theorem says exactly that $\operatorname{Sym}(\mathfrak{g}) \simeq gr_\mathcal{F} \, U$. Since $GK(\operatorname{Sym}(\mathfrak{g}))=\operatorname{Krull} \, \operatorname{Sym}(\mathfrak{g})=\dim \, \mathfrak{g}$, $GK \, U(\mathfrak{g})= \dim \, \mathfrak{g}.$
\end{example}

\begin{proposition}\label{good-filtrations}
    Let $(A,\mathcal{F})$ be a filtered algebra and $M$ be an $A$-module. The following statements are equivalent:
    \begin{enumerate}[{\rm (a)}]
    \item $M$ is a finitely generated $A$-module.
    \item $M$ admits a filtration $\Omega$ such that $\gr_\Omega \, M$ is a finitely generated $\gr_\mathcal{F} \, A$-module. In this case we say that $\Omega$ is a \emph{good filtration}.
    \end{enumerate}
\end{proposition}

\begin{proof}
    \cite[Lemma 6.7]{KL}.
\end{proof}

\begin{definition}
    Let $A$ be a filtered algebra and $M$ be an $A$-module. Two filtrations $\Omega=\{ M_i \}_{i \geq 0 }$ and $\Sigma= \{ N_i \}_{i \geq 0}$ on $M$ are called \emph{equivalent} if there is a natural number $n$ such that for all $i\in\mathbb{N}$,
    \[N_i \subset M_{i+n} \quad\text{and}\quad M_i \subset N_{i+n}.\]
\end{definition}

\begin{proposition}
    Let $A$ be an algebra with a finite-dimensional filtration and $M$ an $A$-module. Two finite-dimensional filtrations on $M$ are equivalent if their associated graded modules are finitely generated over the associated graded algebra of $A$.
\end{proposition}

\begin{proof}
    \cite[Corollary 6.12]{KL}.
\end{proof}

\section{Very nice and modest algebras} \label{sec:verynice}

\subsection{Definitions and First Results} 

Let $\AMod$ denote the category of all $A$-modules, and $\Amod$ the full subcategory whose objects are the finitely generated $A$-modules.

\begin{definition}
    Let $A$ be an affine algebra. We say that \emph{$\GK$ is exact on $\Amod$}
    \footnote{In \cite{KL}, the authors say \emph{Gelfand-Kirillov dimension is exact for finitely generated modules}.} if for any short exact sequence of finitely generated $A$-modules
    \[ 0 \rightarrow M' \rightarrow M \rightarrow M'' \rightarrow 0\]
    we have 
    \[\GK(M)=\max\{ \GK(M'), \GK(M'') \}.\]
\end{definition}

In general, we only have $\GK(M) \geq \max \{\GK(M'),\GK(M'')$\}, by \cite[Proposition 5.1c)]{KL}. The two most general situation where $\GK$ is exact on $\Amod$ are given by theorems of Tauvel and Lenegan.

\begin{theorem}[Tauvel]
Let $A$ be an affine algebra with finite-dimensional filtration $\mathcal{F}$ such that the associated graded algebra $\gr_\mathcal{F} \, A$ is finitely generated Noetherian. Then $\GK$ is exact on $\Amod$.
\end{theorem}

\begin{proof}
   \cite{Tauvel}, \cite[Theorem 6.14]{KL}.
\end{proof}

\begin{theorem}[Lenagan]
    Let $A$ be a Noetherian affine PI-algebra. Then $\GK$ is exact on $\Amod$.
\end{theorem}
\begin{proof}
    \cite[Lemmas 10.13, 10.14, Theorem 10.15]{KL}, \cite{Lenagan}.
\end{proof}

It is still an open problem whether $\GK$ is exact on $\Amod$ for every Noetherian algebra $A$.

In the following, we introduce an abstract class of algebras which have the common features of many of the algebras discussed in \cite{KL}, \cite{Lorenz}, \cite{BD} and \cite[Chapter 8]{McConnell}.

\begin{definition}\label{very-nice}
    An affine algebra $A$ is called \emph{very nice} if
    \begin{enumerate}[{\rm (i)}]
    \item $\GK$ is exact on $\Amod$,
    \item $\GK(M)$ is a non-negative integer for every finitely generated $A$-module $M$, 
    \item there is a function $e: \Amod \rightarrow \mathbb{N}$, called \emph{multiplicity}, satisfying the following properties for any short exact sequence $0\to M'\to M\to M''\to 0$ of finitely generated $A$-modules:
    \begin{enumerate}[{\rm (a)}]
    \item if $\GK(M')< \GK(M)=\GK(M'')$, then $e(M)=e(M'')$;
    \item if $\GK(M')=\GK(M)>\GK(M'')$, then $e(M')=e(M)$;
    \item if $\GK(M')=\GK(M)=\GK(M'')$, then $e(M)=e(M')+e(M'')$;
    \item $e(M)=0$ if and only if $M=0$.
    \end{enumerate}
    \end{enumerate}
    If we require only that the multiplicities $e(M)$ belong to $\mathbb{Q}_{\ge 0}$, we will call the algebra \emph{modest}.
\end{definition}

The notion of multiplicity appeared first in commutative algebra, in the work of David Hilbert, and is a very useful tool in this subject (see \cite[Section 14]{Matsumura}).
The above conditions were explored in a more restrictive setting in \cite[Section~3.1]{FS4}.

\begin{example}
    Let $A$ be a finite dimensional algebra. Then $\Amod$ is the cateogry of finite dimensional modules. Conditions (i) and (ii) in the definition of very nice algebras are trivially satisfied, since for every $M \in \Amod$, since it is finite dimensional, $\GK \, M=0$. We can define a multiplicity function on a module $M$ by setting $e(M):=\dim M$.
\end{example}

\begin{proposition}\label{DCC}
    Let $A$ be a very nice algebra with multiplicity function $e$, and let $M$ be a finitely generated $A$-module. Put $d=\GK(M)$. Let
    \[ M=M_0 \supsetneq M_1 \supsetneq \ldots \supsetneq M_n \]
    be a strictly decreasing chain of submodules with $\GK(M_i/M_{i+1})=d$, $0\le i<n$. Then:
    \begin{enumerate}[{\rm (a)}]
    \item  $n \leq e(M)$.
    \end{enumerate}
\end{proposition}

\begin{proof}
(a) $e(M) = \sum_{j=0}^{i-1}e(M_j/M_{j+1})+e(M_n)=e(M_0/M_1) + e(M_1/M_2) + \cdots + e(M_n)$. Each of these summands is at least 1. So $n \leq e(M)$.
\end{proof}

That is, as in  \cite[8.3.17]{McConnell}, very nice algebras are finitely partitive:

\begin{definition}\label{partitive}
An affine algebra $A$ with finite Gelfand-Kirillov dimension is called \emph{finitely partitive} if, given a finitely generated $A$-module $M$, it has integer Gelfand-Kirillov dimension and there exists $n>0$ such that, for every chain
\[ M=M_0 \supsetneq M_1 \supsetneq M_2 \supsetneq \cdots \supsetneq M_m\] with $\GK(M_i/M_{i+1})=\GK(M)$, we have $m \leq n$.
\end{definition}

The same can be shown for modest algebras.

\begin{proposition}
    Let $A$ be a modest algebra. Then it is finitely partitive.
\end{proposition}

\begin{proof}
    \cite[Lemma 3.2]{BD}.
\end{proof}

It is an open problem if there exists algebras with finite Gelfand-Kirillov dimension which are not finitely partitive.

\begin{proposition}
    Let $A$ a finitely partitive algebra with finite Gelfand-Kirillov dimension. Then $\mathcal{K}(A) \leq \GK(A)$ and, for every finitely generated module $M$, $\mathcal{K}(M) \leq \GK(M)$ --- where $\mathcal{K}(\cdot)$ is the Krull dimension in the sense of Gabriel-Rentschler (cf. \cite[Chapter 6]{McConnell}).
\end{proposition}

\begin{proof}
    \cite[\S 8.3.18]{McConnell}.
\end{proof}

\subsection{Algebras admitting Hilbert-Samuel polynomials}

\begin{definition}\label{definition-Hilbert-Samuel}
    Let $A$ be an affine algebra with a finite dimensional filtration $\mathcal{F}_A$ such that $\gr_{\mathcal{F}_A} \, A$ is affine and Noetherian. We say that $A$ \emph{admits Hilbert-Samuel polynomials} if for every $M \in \Amod$, and every finite dimensional good filtration $\Omega=\{ M_i \}_{i \geq 0}$ on it, there exists a polynomial $H_{M, \Omega}(x): \mathbb{N} \rightarrow \mathbb{R}$ with rational coeficients such that $H_{M, \Omega}(n)=\dim \, M_n$ for $n\gg 0.$
\end{definition}

\begin{theorem}\label{theorem-Hilbert-Samuel}
Let $A$ be an affine algebra that admits Hilbert-Samuel polynomials. Then it is very nice.
\end{theorem}
\begin{proof}
    Let $0 \rightarrow M' \rightarrow M \rightarrow M'' \rightarrow 0$ be a short exact sequence in $\Amod$. Let $\Omega=\{ M_i \}_{i \geq 0}$ be a finite dimensional good filtration for $M$, and define finite dimensional filtrations $\Omega'=\{ M_i \cap M' \}_{i \geq 0}$ for $M'$ and $\Omega''=\{ \phi(M_i) \}_{i \geq 0}$ for $M''$, where $\phi: M \rightarrow M''$ in the above short exact sequence.

    Then we have a short exact sequence
    \[ 0 \rightarrow \gr_{\Omega'} M' \rightarrow \gr_\Omega M \rightarrow \gr_{\Omega''} M'' \rightarrow 0. \]
Since $gr_{\mathcal{F}_A} \, A$ is Noetherian and $\gr_\Omega M$ is finitely generated, it is Noetherian. Hence, $\gr_{\Omega'} M'$ and $\gr_{\Omega''} M''$ are Noetherian, and hence $\Omega'$ and $\Omega''$ are good filtrations by Proposition \ref{good-filtrations}. We have $\dim M_n = \dim (M_n \cap M') + \dim \phi(M_n)$, for every $n \in \mathbb{N}$. Hence, if $n\gg 0$, then $H_{M, \Omega}(n)=H_{M',\Omega'}(n)+H_{M'', \Omega''}(n).$ By Proposition \ref{convenient-proposition}, the coefficient of the leading term of each of these polynomials is positive, so $\deg(H_{M,\Omega})=\max \{ \deg(H_{M',\Omega'}), \deg(H_{M'', \Omega''}) \}$. Now, by Proposition \ref{growth}(ii)(iii),
\[\GK(M)=\deg(H_{M,\Omega})=\max\{ GK(M')=\deg(H_{M',\Omega'}), GK(M'')=\deg(H_{M'', \Omega''}) \}.\] Hence conditions (i) and (ii) of the definition of a very nice algebra are satisfied. It can be shown that the degree and Bernstein numbers of the Hilbert-Samuel functions of a module $M$ are independent of the chosen filtration (\cite[Chapter 7]{KL}). Hence we can take $e(M)$ to be the Bernstein number of any Hilbert-Samuel polynomial for $M$. This clearly satisfies condition (iii) in the definition of a very nice algebra. Hence we are done.
\end{proof}

\begin{example}
    The existence of Hilbert-Samuel polynomials is usually obtained when the graded associated algebra $\gr_{\mathcal{F}_A} \, A$ is affine commutative. But they can also be shown to exist if $\gr_{\mathcal{F}_A}$ is an enveloping algebra \cite[Chapter 9]{Jantzen} or a quantum affine space \cite{MP}.
    
\end{example}
\begin{proposition}
    Let $G$ a complex reflection group representation $h$, and denote by $H_c(G,h)$ the rational Cherednik algebra (cf. \cite{EG}) when $t=1$. Then $H_c(G,h)$ admits Hilbert-Samuel polynomials, and hence is a very nice algebra. It is also a prime ring.
\end{proposition}

\begin{proof}
    That $H_c(G,h)$ admits Hilbert-Samuel polynomials is \cite[2.3]{Thompson}. Hence by Theorem \ref{theorem-Hilbert-Samuel}, it is a very nice algebra. It is a prime ring by \cite[Corollary 1.3.2(1)]{Bellamy0}.
\end{proof}

\subsection{Holonomicity}

Now we need to discuss the notion of \emph{holonomic} modules. They were initally introduced by J. Bernstein for the Weyl algebra. He proved that if $M$ is a finitely generated module over $W_n(k)$, then $\GK(M) \geq n$. This is called the \emph{Bernstein inequality}, and it was generalized for rings of differential operators on smooth affine varieties (for these rings, see \cite{Bjork}, \cite[Chapter 15]{McConnell}). If $X$ is such a variety, and $M$ a finitely generated $\mathcal{D}(X)$-module, then $\GK(M) \geq \dim \, X$.

In both cases, the modules of minimal Gelfand-Kirillov dimension are called holonomic. They constitute a very important subcategory of $\mathcal{D}$-modules \cite{Hotta}, \cite{Borel}.

Then, in an unpublished result, Gabber's proved the following inequality:
\begin{theorem}[Gabber]
    Let $\mathfrak{g}$ be an algebraic Lie algebra over an algebraically closed field of 0 characteristic, and $M$ a finitely generated module. Then
    \[
        \GK \, M \geq \frac{1}{2} \GK \big(U(\mathfrak{g})/\Ann(M)\big),
    \]
    and equality holds for modules in the category $\mathcal{O}$ and Harish-Chandra modules in the sense \cite[Chapter 9]{Dixmier}.
\end{theorem}

\begin{proof}
    See \cite{Jantzen} or \cite{French}.
\end{proof}

\begin{remark}
    If the Lie algebra is not algebraic, the inequality may fail \cite{Stafford}.
\end{remark}
Analogues of holonomic modules are studied for many classes of algebras: symplectic reflection algebras \cite{Losev}, rational Cherednik algebra on arbitrary varieties \cite{Bellamy}, etc. 

\begin{definition}[{\cite[8.5.8]{McConnell}}]\label{holonomic-1}
    A finitely generated $A$-module $M$ is called \emph{holonomic} if
    \[ \GK(M)= \frac{1}{2} \GK\big(A/\Ann(M)\big).\]
\end{definition}

We adopt a different point of view, influenced by \cite{BF}. Our exposition is inspired by \cite{FS4}.

\begin{definition}\label{holonomic-2}
\phantom{X}
    \begin{enumerate}[{\rm (i)}]
    \item The \emph{holonomic number} of an algebra, $A$, is the smallest number in the set $\big\{\GK(M)\mid M \in \Amod \big\}$, where $\Amod$ is the category of finitely generated modules.
    \item If $A$ is a modest algebra, a finitely generated module $M$ is called \emph{min-holonomic} if $\GK(M)=h_A$. The class of all min-holonomic modules is denoted $\mathcal{H}(A)$.
    \end{enumerate}
\end{definition}

It is an interesting question to study the relations between the notions of holonomic and min-holonomic modules for modest algebras. To the best of our knowledge, the two notions coincide in all known cases, such as D-modules, and it is easy to see that they are equal if the algebra is simple.

\begin{theorem}\label{very-nice-1}
    Let $A$ be a modest algebra.
    \begin{enumerate}[{\rm (a)}]
    \item $\mathcal{H}(A)$ is an abelian category.
    \item If $A$ is Noetherian, any min-holonomic module $M$ has finite length.
    \item If $A$ is very nice and if $e(M)=1$, the module is irreducible.
    \item If $A$ is simple, Noetherian but not Artinian, every min-holonomic is cyclic.
    \end{enumerate}
\end{theorem}

\begin{proof}

If $M$ is a finitely generated $A$-module, every submodule and every homomorphic image have less then or equal Gelfand-Kirillov dimension then $M$. Hence, the kernel and cokernel of a short exact sequence of min-holonomic modules is min-holonomic, an hence (a) follows. If $A$ is Noetherian, it satisfies the ACC condition on finitely generated modules. But by Proposition \ref{DCC}, it also satisfies the DCC condition for holonomic modules. Hence, (b) follows. (c) is a immediate consequence of (b). Finally, (d) follows from \cite[Theorem 10.2.5]{Coutinho}.
\end{proof}

\begin{proposition}
    $\mathcal{H}(A)$ is a thick abelian subcategory of the category of finitely generated $A$-modules.
\end{proposition}

\begin{proof}
    Let $0 \rightarrow M' \rightarrow M \rightarrow M'' \rightarrow 0$. Since the Gelfand-Kirillov dimension is exact, it is clear that $\GK(M)=h_A$ if and only if $\GK(M')=\GK(M'')=h_A.$
\end{proof}

We recall that a left module $M$ for a ring $R$ is called a torsion module if every $0 \neq m \in M$ has torsion: there existis $r \neq 0 \in R$ with $r.m=0$

\begin{proposition}\label{torsion}
    If $\GK(A)>h_A>0$, then every min-holonomic module is a torsion module.
\end{proposition}

\begin{proof}
    Let $M \in \mathcal{H}(A)$, $0 \neq m \in M$. Consider the map $\phi: A \rightarrow M$, $a \mapsto am$. $Im \, \phi$ is a submodule of $M$, and hence min-holonomic. We have the short exact sequence

    \[ 0 \rightarrow \Ker \, \phi \rightarrow A \rightarrow \im \, \phi \rightarrow 0.\]
    As $\GK(A)>h_A$ is the the biggest between $\GK(\ker \, \phi)$ and $\GK(\im \, \phi)=h_A$, we conclue that $\GK(\ker \, \phi)=\GK(A)>h_A$ and hence it is not the 0 module. Hence $m$ has torsion.
\end{proof}

\begin{proposition}\label{GK-quotient}
    Let $A$ be an affine algebra which is a prime ring and has positive Gelfand-Kirillov dimension. If $I$ is a proper non-null left ideal of $A$, $\GK \, A/I \leq \GK \, A -1$. In particular, this result holds if $A$ is a Noetherian domain.
\end{proposition}
\begin{proof}
    First we notice that every proper left ideal $I$ of $A$ is essential. Let $J$ be another proper left ideal. If $I \cap J = (0)$, then $IJ \subset I \cap J = (0)$. But as $A$ is prime, this implies $I=(0)$ or $J=(0)$, a contradiction. Hence the result follows from \cite[8.3.5(i)]{McConnell}. If $A$ is a Noetherian domain, by Proposition \ref{it-is-Ore}, then it is an Ore domain, and hence prime (\cite[Chapter 2]{McConnell}).
\end{proof}

\begin{proposition}\label{GK-quotient-2}
    Let $A$ be a simple 
    ring with positive Gelfand-Kirillov dimension and a modest algebra. If $I$ is non-zero proper left ideal, then $\GK \, A/I \leq \GK \, A -1.$
\end{proposition}

\begin{proof}
    Suppose first that $I$ is a principal left ideal $Aa$ with $a$ a regular element. Consider the right exact sequence:
    \[0 \rightarrow A \rightarrow A \rightarrow A/Aa \rightarrow 0,\]
    where the first map is given by multiplication by $a$. If $\GK \, A = \GK \, A/Aa$, then $e(A)=e(A)+e(A/Aa)$, that is, $e(A/Aa)=0$, which is absurd. In the general case, $I$ contains an ideal of the form $Aa$, with $a$ a regular element, by Goldie's regular element lemma \cite[Proposition 6.3]{GW} and the fact that $I$ is an essential ideal (cf. proof of Proposition \ref{GK-quotient}), and $A/I$ is a homomorphic image of $A/Aa$, implies our claim.
\end{proof}

The following is a strengthening of \cite[Theorem 3.14]{FS4}:

\begin{proposition}\label{formal-proposition}
    Let $A$ be an infinite-dimensional affine algebra which is also an prime ring (for instance, an Ore domain). Suppose $h_A=\GK(A)-1$. A finitely generated module $M$ is min-holonomic, if and only if, it is torsion. Moreover, if $A$ is simple, $\GK(A)=h_A$, and every irreducible module is min-holonomic.
\end{proposition}

\begin{proof}
     By Proposition \ref{torsion}, a simple f.g. module $M$ is a torsion module. Now suppose $M$ f.g. torsion. Assume it is generated by $m_1, \ldots, m_n$. Since each $m_i$ is torsion, $Am_i$ is isomorphic to a quotient of the form $A/J_i$, where $J_i$ is a proper left ideal. Hence $\GK \, A/J_i=h_A$ by Proposition \ref{GK-quotient}. As the category of min-holonomic modules is abelian, each $Am_i \in \mathcal{H}(A)$, and hence $M=\sum_ Am_i \in \mathcal{H}(A)$. If $A$ is simple, $\GK(M) \geq 1$, for $M$ is necessarily infinite-dimensional. The simple finitely generated $A$-modules have the form $A/I$, where $I$ is a left-maximal ideal. Applying Proposition \ref{GK-quotient}, $\GK(A/I)=h_A$, and hence the module is min-holonomic.
\end{proof}

\begin{example}
    Every irreducible module for the first Weyl algebra is holonomic.
\end{example}

Finally, we have the following technical improvement of the last proposition:

\begin{proposition}\label{the-real-deal}
Let $A$ be an affine infinite-dimensional algebra which is a prime ring, but not a division algebra, with $\GK \, A=2$. If $M$ is an irreducible infinite-dimensional module, $\GK \, M =1$
\end{proposition}

\begin{proof}
    $M$ is an irreducible module, $M \simeq A/I$ for a suitable proper left ideal $I$ of $A$. By Proposition \ref{GK-quotient}, $\GK \, M =0, 1$. Since $M$ is infinite-dimensional, $\GK \, M=1$.
\end{proof}

\section{Algebras with special filtrations}
\label{sec:algebras-with-special-filtrations}

In this section, we will consider affine algebras $A$, and many classes of special filtrations on them, such that the algebra is one of the following four: a) filtered-semicommutative, b) almost commutative, c) somewhat commutative, d) vaguely commutative (a similar notion considered in \cite{BD} and \cite{FS4}).

We will first discuss almost commutative algebras. As this is the most important example in applications (\cite{KL}), a very detailed account is given. To a lesser extent, we tried to do the same for somewhat commutative algebras, but we had to omit some steps, as the theory is very technical. Vaguely commutative algebras and filtered semi-commutative algebras have a brief exposition.

What makes these algebras so relevant, in all cases, is the same: the existence of a Hilbert -Samuel polynomials. So the results, in the end, will follow from  Theorem \ref{theorem-Hilbert-Samuel}.

We remark that those algebras also have the notion of a Poincaré series. The relation of the Gelfand-Kirillov dimension and Poincaré series is the main topic of \cite{Lorenz}. We will dedicate the last subsection of this section to discuss them.

We have 
\[\text{almost commutative $\subsetneq$ somewhat commutative $\subsetneq$ vaguely commutative.}\]

\subsection{Almost commutative algebras}

\begin{definition}[{\cite[Ch. 7]{KL}}]
An affine algebra $A$ is called \emph{almost commutative} if it has a filtration $\mathcal{F}=\{ F_i \}_{i \geq 0}$, with $\mathcal{F}_0=k$, $F_1$ is finite-dimensional, $F_n=F_1^n$, $n>1$, and $\gr_\mathcal{F} \, A$ is  commutative. Such filtrations are sometimes called \emph{standard}, and they are finite-dimensional.
\end{definition}

\begin{proposition}
    If $A$ is almost commutative and $\mathcal{F}$ is a standard filtration on $A$, then 
    $\gr_\mathcal{F} \, A$ is a commutative finitely generated Noetherian algebra. Hence $A$ is Noetherian.
\end{proposition}

\begin{proof}
    \cite[Proposition 7.1]{KL}
\end{proof}

Almost commutative algebras may seem like an artificial generalization of commutative algebras. But, in fact:

\begin{theorem}
    An algebra $A$ is almost commutative if and only if it is a homomorphic image of a universal enveloping algebra $U(\mathfrak{g})$ of a finite-dimensional Lie algebra $\mathfrak{g}$.
\end{theorem}

\begin{proof}
    \cite[Theorem 7.2]{KL}.
\end{proof}

\begin{example}\label{Bernstein-filtration}
    It is well-known that the Weyl algebra $W_n(k)$ is generated by monomials $x^\alpha y^\beta$ in the standard generators (here we are using the usual multi-index notation). We are now going to introduce the important \emph{Bernstein filtration} $\mathcal{B}=\{ B_i \}_{i \geq 0}$ on $W_n(k)$. Here $B_i$ is the vector space spanned by the monomials $x^\alpha y^\beta$ with $|\alpha|+|\beta| \leq i$. Elementary combinatorics shows that each $B_i$ is finite-dimensional and $B_i=B_1^i$. Finally, $\gr_\mathcal{B} \, W_n(k)= k[z_1, \ldots, z_{2n}]$ (\cite{Coutinho}). So the Weyl algebras are almost commutative, and by Proposition \ref{Proposition 6.6}, $GK \, W_n(k)=2n$
\end{example}

We could also show that $W_n(k)$ is almost commutative by noticing it is an homomorphic image of the enveloping algebra of the $n$:th Heisenberg Lie algebra $\mathfrak{h}_n$ which has a basis $\{a_i,b_i\}_{i=1}^n\cup\{c\}$ with $[a_i,b_i]=c$ and remaining brackets zero.

\begin{proof}
    \cite[Lemma 1.5]{KL}.
\end{proof}

\begin{theorem}[Hilbert polynomials]\label{Hilbert}
    Let $A=k[x_1, \ldots, x_r]$ by graded by the usual degree, and let $M=\bigoplus_{i=0}^\infty M_i$ be a finitely generated graded $A$-module.
\begin{enumerate}[{\rm (a)}]
    \item Each $M_i$ is finite-dimensional.
    \item  There is a polynomial $h_M(x)\in\mathbb{Q}[x]$ of degree at most $r-1$ such that for $n\gg 0$, $\dim M_n=h_M(n)$.
    
    \item For $n\gg 0$, $d_M(n)=\dim(\bigoplus_{i=0}^n M_i)$ is a polynomial of degree $\leq r$ and rational coefficients.
\end{enumerate}
\end{theorem}

\begin{proof}
    \cite[Theorem 7.4]{KL}. Note that (b) implies (c) by Proposition \ref{convenient-proposition}.
\end{proof}

\begin{proposition}\label{basic}
    Let $A=\bigoplus_{i=0}^\infty A_i$ be a commutative graded algebra with $A_0=k$ and finitely generated as an algebra by the subspace $A_1$, which is finite-dimensional, and let $M=\bigoplus_{i=0}^\infty M_i$ be a finitely generated graded $A$-module.

    Let $d_M(n)=\dim(\bigoplus_{i=0}^n M_i)$; it is a polynomial in $n$ with rational coefficients and degree $\GK \, M$.

\end{proposition}

\begin{proof}
    $A$ is an homomorphic image of $Sym(A_1)$, which is a polynomial algebra in $dim \, A_1$ indeterminates. In this way $M$ becomes a finitely generated graded module over a polynomial algebra, and hence the first assertion follows from Theorem \ref{Hilbert}. By \cite[Proposition 5.1c)]{KL},
    $\GK(M)$ as an $A$-module is the same as $\GK(M)$ as an $Sym(A_1)$-module. By \cite[Proposition 6.1b)]{KL} and Theorem \ref{Hilbert}c), so $\GK(M)=\gamma(d_M(n))$, which is precisely the degree of $d_M(n)$ by Proposition \ref{growth}.
\end{proof}

Let $A$ be an algebra with a filtration $\mathcal{F}$ that turns it into an almost commutative algebra, and $M$ an $A$-module with a finite-dimensional filtration $\Omega$ such that $\gr_\Omega \, M$ is a finitely generated $\gr_\mathcal{F} \, A$-module. By Proposition \ref{basic}, for $n\gg 0$, $d_{\gr_\Omega \, M}(n)=\dim(M_0 \oplus M_1/M_0 \oplus M_2/ M_1 \oplus \ldots \oplus M_n/M_{n-1})=\dim M_n= d_\Omega(n)$ is a polynomial in $n$ with rational coefficients called the \emph{Hilbert-Samuel polynomial}.

The polynomial $d_\Omega(n)$ can be written, using Proposition \ref{convenient-proposition} as

\[d_\Omega(n)= a_d \binom{n}{d} + a_{d-1} \binom{n}{d-1}+ \ldots a_1 \binom{n}{1}+a_0,\]

where $e(M)=a_d$ is a positive integer, called the \emph{multiplicity} of $M$.

\begin{remark}
    We can also define $e(M)$ as (the leading coeficient of $d_\Omega(n)$) $\times$ $(\GK(M))!$.
\end{remark}

By Proposition \ref{Proposition 6.6}, the degree of $d_\Omega(n)$ is $\GK(M)$.

By Theorem \ref{theorem-Hilbert-Samuel}, we have:

\begin{theorem}\label{almost-commutative}
    An almost commutative algebra is very nice algebra. 
    \end{theorem}

\begin{example}
    Let $M$ be the Weyl algebra $W_n(k)$ itself filtered by the Bernstein filtration. $\dim \, B_j= \binom{2n+j}{2n} = \frac{j^{2n}}{(2n)!}$. It is a polynomial of degree $2n$, which implies $\GK \, W_n(k)=2n$, in accordance with Example \ref{Bernstein-filtration}. The multiplicity is the leading coeficient times $(2n)!$ - hence, $1$.
\end{example}

\begin{example}

Let $M=k[x_1,\ldots,x_n]$, with its usual filtration. $M$ is a module for the Weyl algebra $W_n(k)$ and the usual filtration of $M$ is compatible with the Bernstein filtration. $\dim M_j = \binom{n+j}{n}=\frac{j^n}{n!}+ \ldots$, a polynomial in $j$ with degree $n$. Hence $\GK \, M = n$. The multiplicity is the coefficient of the leading term, $1/n!$, times $\GK(M)!$, and hence is 1.
\end{example}

\begin{example}
    Let $M=k[x,x^{-1}]$, with filtration given by $M_j=B_J x^{-1}$. One can see that $M_j$ is spanned by $x^{j-1}, x^{j-2}, \ldots, x^{-j-1}$. So $\dim M_j = 2j +1$. $\GK(M)=1$, as the polynomial has degree 1, and the multiplicity is $2$.
\end{example}

\subsection{Somewhat commutative algebras}

\begin{definition}
    An affine algebra $A$ is called \emph{somewhat commutative} if has a finite-dimensional filtration $\mathcal{F}$ with $F_0=k$ and $\gr_\mathcal{F} \, A$ is affine commutative.
\end{definition}

So the difference between almost commutative algebras and somewhat commutative algebras is that, in the latter case, we do not require that the algebra is generated in degree $1$.

\begin{example}\label{diff-ops-filtration}
    Every almost commutative algebra is a somewhat commutative algebra, but the converse is not true: for instance, the algebra $k[y][x;-y^2 \partial_y$] is somewhat commutative but not almost commutative (\cite[14.3.9]{McConnell}.)
\end{example}

\begin{example}
    Let $A$ be an affine regular algebra. Then $\mathcal{D}(A)$ is a somewhat commutative algebra \cite[Theorem 15.1.20 a)i)ii)]{McConnell}. The associated graded algebra is $\operatorname{Sym}(\operatorname{Der}_k(A))$. Since $A$ is regular, $\GK \, \mathcal{D}(A)= \GK \, \operatorname{Sym}(\operatorname{Der}_k(A)) = 2 \, \GK \, A = 2 \operatorname{Krull} A$.
\end{example}

\begin{example}\label{invariants-weyl-algebra-filtration}
    Consider $W_n(k)$ the rank n Weyl algebra with Bernstein filtration $\mathcal{B}$. As we saw, this turn $W_n(k)$ in an almost commutative algebra. Let $G<GL_n(k)$ be a finite group that acts linearly on $W_n(k)$. Then $W_n(k)^G$ has an inherited filtration $\mathcal{B}^G=\{ B_i^G \}_{i \geq 0 }$. Let $h$ be the natural $n$-dimensional representation of $G$ as given. Then $\gr_{\mathcal{B}^G} \, W_n(k)^G=k[h \oplus h^*]^G$. Since the later algebra is commutative and affine (by Noether's Theorem in invariant theory), the invariants of the Weyl algebra are a somewhat commutative algebra. Also, as the associated graded algebra has Gelfand-Kirillov dimension $2n$, $\GK \, W_n(k)^G = 2n$.
\end{example}

\begin{example}
    The previous example can be generalized considerably. Consider the following alternative realization of the Weyl algebra: Let $(V, \omega)$ be a symplectic vector space. Let $\dim \, V = 2n $. Consider the quotient of the free algebra $k \langle V \rangle$ by the ideal generated by $xy-yx - \omega(x,y), \, x,y \in V$. This is the Weyl algebra $W_n(k)$. Let $G$ be a finite group of automorphisms of $Sp(V)$. Then $G$ acts on the Weyl algebra $W_n(k)$ and preserves the Bernstein filtration.
    The associated graded algebra of $W_n(k)^G$ is $S(V)^G$, and so $W_n(k)^G$ is somewhat commutative. The previous case is a subcase of this: let $h$ be a finite-dimensional vector space of dimension $n$ and $G<GL(h)$ finite. Then $G$ acts on $h^*$ contragrediently and on $h \oplus h^*$ diagonally. The vector space $V=h \oplus h^*$ has a canonical symplectic form, and so the previous example is a case of this one. The symplectic form is: $\omega\langle (y,f), (u,g) \rangle=g(y)-f(u), \, f,g \in h^*, y,u \in h$.
\end{example}

\begin{example} Let $A$ be an affine regular domain, and $G$ a finite group of automorphisms of $A$. Then $\mathcal{D}(A)^G$ is a somewhat commutative algebra \cite[Proposition 5.2]{FS4}. The  associated graded algebra is $\operatorname{Sym}(A)^G$, which has Gelfand-Kirillov dimension $2 \, \GK(A)$, and by Proposition \ref{Proposition 6.6}, $\GK \, \mathcal{D}(A)^G= 2 \, \GK(A)$.
\end{example}

We denote $\mathcal{F}=\{ A_i \}_{i \geq 0}$. If $M$ is a module with a finite-dimensional filtration $\Omega$, we write $\Omega=\{ M_i \}_{i \geq 0}$. This deviate from the notation of the previous subsection but will be very convenient here.

Call $R=\gr_\mathcal{F} \, A$ and $N= \gr_\Omega \, M$, where $\Omega$ is a good filtration (i.e., the associated graded modules is finitely generated over  $R$)

$R$ and $N$ are obviously $\mathbb{N}$-graded, say $R=\bigoplus_{i=0}^\infty R_i$; and the same for $N.$ But we can use this to define a filtration  $\mathcal{R}=\{ R_i^* \}_{i  \geq 0}$, where $R_n^* = \oplus_{i=0}^n R_i$. Similarly we construct a filtration $\mathcal{N}=\{ N_i^* \}_{i \geq o}$.

\begin{lemma}
 There is an $t>0$ such that $R$ is generated, as an algebra, by $R_t^*$, and moreover for every $n$, $R_n^*  \subset (R_t^*)^n  \subset R_{nt}^*$. Also $A_n \subset (A_t)^n \subset A_{nt}$. A similar result holds for $N$.  
\end{lemma}

\begin{proof}
    \cite[Lemma 8.6.2]{McConnell}.
\end{proof}

\begin{lemma}\label{lemma-sca-1}
    Let $M$ be a finitely generated $A$-modules. Let $\Omega$ be a good filtration for $M$. Let $t$ be as in the previous lemma, and there is an $s>0$ such that $N_s^*$ generates $N$ as an $R$-module. Then for each $n$:
    \[ N_n^* \subset R_n^*N_s^* \subset (R_t^*)^n N_s^* \subset R_{nt}^*N_s^* \subset N_{s+nt}^*, \]
    and
    \[ M_n \subset A_n M_s \subset (A_t)^n M_s \subset A_{nt}M_s \subset M_{s+nt}. \]
\end{lemma}

\begin{proof}
    Similar to that of the previous lemma.
\end{proof}

An immediate consequence of this fact is the following.

\begin{proposition}
    $\GK(A)=\GK(R)$, and if $M$ is a finitely generated module with good filtration $\Omega$, $\GK(M)=\GK(N)$.
\end{proposition}

Now notice that $R$ is commutative affine. Hence all the results from almost commutative algebras are applicable. We can construct a filtration for $R$ such that it becomes almost commutative, with Hilbert polynomial $p(t)$ which in turn gives $\GK(N)$ and $e(N)$. It is now needed to show that $G(N)$ does not depends of the filtrations for $A$ and $M$; and that $e(N)$ does not depend on the filtrations for $M$. The rest of \cite[Section 8.6]{McConnell} is too technical, so we omit it.

Again by Theorem \ref{theorem-Hilbert-Samuel}:
\begin{theorem}
    A somewhat commutative algebra is a very nice algebra.
\end{theorem}

The class of somewhat commutative algebras is closed under a relevant algebraic operation: almost centralizing extensions.

\begin{definition}
    Let $R \subset S$ be two algebras. $S$ is called an \emph{almost centralizing extension} of $R$ if it is generated as an algebra by $R$ and a finite set $\{ x_1, \ldots, x_n \} \subset S \setminus R$, such that
    \begin{align*}
    [r,x_i] &\in R, \quad\forall r \in R,\, i=1,\ldots,n; \\
    [x_i,x_j] &\in \sum_{k=1}^n x_k R +R,\quad\forall i=1,\ldots,n.
    \end{align*}
\end{definition}

\begin{example}
    Let $R$ be an affine algebra and $\mathfrak{g}$ a finite-dimensional Lie algebra acting by derivations on $R$. Then the smash product $R \# U(\mathfrak{g})$ is an almost centralizing extension of $R$.
\end{example}

\begin{proposition}
    An algebra that is an almost centralizing extension of a somewhat commutative algebra, is again somewhat commutative.
\end{proposition}

\begin{proof}
    \cite[Proposition 8.6.9]{McConnell}.
\end{proof}

We insisted on the filtration being finite-dimensional. But in fact we have:

\begin{proposition}
If $A$ is an affine algebra with a (non-necessarily finite dimensional) filtration $\mathcal{F}=\{A_i \}_{i \geq 0}$ with $k \subset A_0$, and $M$ an $A$-module filtered by $\Omega$, then:
\[\GK(A)=\GK(\gr_\mathcal{F} \, A)\quad\text{and}\quad \GK(M)=\GK(\gr_\Omega \, M).\]
\end{proposition}

\begin{proof}
    \cite{MStafford}.
\end{proof}

\begin{example}
    Suppose the base field algebraically closed of zero characteristic. Let $X$ be an smooth irreducible affine variety and $\mathcal{D}(X)$ the ring of differential operators on $X$. $\mathcal{D}(X)$ has a infinite dimensional filtration, by order of differential operator, such that the associated graded algebra is $\mathcal{O}(T^*X)$. Since $\GK(\mathcal{O}(T^*X))=2 \dim X$, by the above proposition $\GK(\mathcal{D}(X)=2 \dim X$.
\end{example}

This result seems at odd with Proposition \ref{Proposition 6.6}, but the idea of the paper \cite{MStafford} is to somehow reduce the given filtration to a finite-dimensional one. This is the process of re-filtering, which we are going to see further below.

\subsection{Vaguely commutative algebras}

We introduce the following terminology:

\begin{definition}\label{vaguely-commutative}

Let $A$ be an affine algebra. If $A$ has a filtration $\mathcal{F}$ that satisfies all conditions of a somewhat commutative algebra, \emph{except} that we allow $k \subsetneq F_0$, then it is called \emph{vaguely commutative algebra}.
\end{definition}

\begin{remark}
The notion of vaguely commutative algebras were investigated under several additional assumptions in \cite[Section~6.1]{FS4}, under the name \emph{generalized somewhat commutative algebra}.
\end{remark}

\begin{theorem}\label{generalized somewhat commutative algebra}
     All vaguely commutative algebras are modest algebras.
\end{theorem}
\begin{proof}
    \cite[Theorem 3.2, Proposition 3.3]{BD} and \cite[Lemma 2.1, Proposition 6.6]{KL}.
\end{proof}

In general, for a vaguely commutative algebra  the multiplicity of a module does not belong to $\mathbb{N}$. It is of the form $n/\ell^{\GK(M)}$, where $n \in \mathbb{N}^+$ and $\ell$ is the least common multiple of the degree of the homogeneous generators of $\gr_\mathcal{F} \, A$.

We have a fundamental example of this situation. Let $H_c(V,\Gamma)$ be a symplectic reflection algebra at $t=1$ and $U_c(V,\Gamma)$ its spherical subalgebra (the canonical reference for these is \cite{EG}).

We have a filtration for $H_c(V, \Gamma)$ given by $F_{-1}=0$, $F_0=\mathbb{C} \Gamma$, $F_1=\mathbb{C}\Gamma + \mathbb{C} \Gamma V$. $F_n=F_1^n$, $i \geq 2$. Let $\mathcal{G}$ be the filtration induced on $U_c(V,\Gamma)$.

\begin{proposition}
    $\gr_\mathcal{G} U_c(V,\Gamma) \simeq S(V)^\Gamma$. Hence, $U_c(V, \Gamma)$ is vaguely commutative.
\end{proposition}
\begin{proof}
    \cite[Theorem 1.3]{EG}.
\end{proof}

\subsection{Filtered semi-commutative algebras}

\begin{definition}
    An affine algebra $A$ is called \emph{semi-commutative} if it is generated by a finite number of elements $x_1, \ldots, x_n$ with $x_i x_j = \lambda_{ij} x_j, x_i$, $0 \neq \lambda_{ij}$, for all $i,j$. That is, $A$ is a homomorphic image of a quantum affine space.
\end{definition}

\begin{definition}
    An affine algebra $A$ is called \emph{filtered semi-commutative} if it has a finite-dimensional filtration $\mathcal{F}$, with $F_0=k$, such that $\gr_\mathcal{F} \, A$ is semi-commutative
\end{definition}

\begin{example}
As example of such algebras we have $\mathcal{O}_q\big(\mathbb{M}_n(k)\big)$, $\mathcal{O}_q\big(GL_n(k)\big)$,
$\mathcal{O}_q\big(SL_n(k)\big)$ and $U_q\big(\mathfrak{sl}_n(k)\big)$, where $q\in k^\times$ is not a root of unity \cite{McConnnel2}.
\end{example}

\begin{theorem}
 Filtered semi-commutative algebras are very nice.
 \end{theorem}
 \begin{proof}
     This is the main result in \cite{McConnnel2}. 
     
 \end{proof}

\subsection{Some comments on Poincaré series}

\begin{definition}
    Let $A=\bigoplus_{i=0}^\infty A_i$ be an algebra with finite-dimensional grading. The Poincaré series of $A$ is the generating function $P_A(t)=\sum_{i=0}^\infty \operatorname{dim} \, A_i t^i$. If $A$ has a finite-dimensional filtration, we denote by $P_A(t)$ the Poincaré-Hilbert series of the associated graded algebra of $A$. 
\end{definition}

It is well known that we have a natural map $\iota: \mathbb{Q}(t) \into \mathbb{Z}((t))$, where $\mathbb{Q}(t)= \operatorname{Frac} \, \mathbb{Z}[t]$, and $\mathbb{Z}((t))=\operatorname{Frac} \, \mathbb{Z}[[t]]$. $P_A(t)$ belongs to $\mathbb{Z}[[t]]$. The fundamental question is: when there are $p(t),0 \neq q(t) \in \mathbb{Z}[t]$ such that $\iota\big((p(t)/q(t)\big)=P_A(t)$? That is, when $P_A(t)$ can be expressed as a rational function?

The next result shows why this is an interesting question.

\begin{theorem}
    Let $P(t)=\sum_{i=0}^\infty f(i) t^i$ be any element of $\mathbb{Z}[[t]]$. The following are equivalent:
    \begin{enumerate}[{\rm (a)}]
    \item $P(t)$ is a rational function; that is, there are $p(t),0 \neq q(t) \in \mathbb{Z}[t]$, $q(0)=1$, with $P(t)=p(t)/q(t)$
    \item There exists $0 \neq s \in \mathbb{N}$ and $a_1, \ldots, a_s \in \mathbb{Z}$, such there is a recurrence relation, for $n\gg 0$:
    \[ f(n+s)=a_1f(n+s-1)+ \ldots + a_s f(n). \]
    \item There exists polynomials $p_j(x) \in \bar{\mathbb{Q}}(x)$, where $\bar{\mathbb{Q}}$ denote the algebraic numbers, and elements $\alpha_j \in \bar{\mathbb{Q}}$, $j=1,\ldots, s$, such that $f(n)=\sum_{j=1}^s p_j(n) \alpha_j^n$, for $n\gg 0$.
    \end{enumerate}
\end{theorem}
\begin{proof}
    \cite[Chapter III, Section 1]{Lorenz0}, \cite[Theorem 4.1.1]{Stanley}\footnote{\cite{Stanley} is not an algebra book: it is about enumerative combinatorics!}.
\end{proof}

The Poincaré series is a polynomial if and only if the algebra is finite-dimensional, and hence has Gelfand-Kirillov dimension 0.

\begin{theorem}\label{rational-Poincaré}
    Let $P_A(t)=\sum_{n=0}^\infty f(n) t^n$ be the Poincaré series of an infinite-dimensional algebra $A$ with a finite-dimensional grading/filtration. If $P_A(t)$ is a rational function then its radius of convergence $r$ is $\leq 1$ and either

    \begin{itemize}
        \item $r<1$ and the algebra $A$ has exponential growth (and hence, in particular, infinite Gelfand-Kirillov dimension), or

        \item $r=1$ and
        \[ P_A(t)=\frac{p(t)}{(1-t^s)^d},\]
        for some polynomial $p(t)$ with $p(1) \neq 0$
\end{itemize}
In the later case there exists polynomials $f_1, \ldots, f_s$ such that $f(n)=f_i(n)$ for $n \equiv i (mod \, s)$ and $n\gg 0$; $max\{ deg(f_i) \}= d-1$, $\GK(A)=d$.
\begin{proof}
    \cite[Chapter III, Section 1]{Lorenz0}, \cite[Proposition 4.1]{Stanley}.
\end{proof}
    
\end{theorem}

The same proof as for Theorem \ref{generalized somewhat commutative algebra} gives us:

\begin{theorem}
    Let $A$ be an affine infinite-dimensional algebra such that its Poincaré series is rational, with radius of convergence 1. Then it is a modest algebra.
\end{theorem}

\subsection{The endomorphism property}

\begin{definition}
    Let $A$ be an algebra over a field $k$. Then we say that $A$ has the \emph{endomorphism property} if $\End_A(M)$ is algebraic over $k$, for every simple $A$-module $M$.
\end{definition}

\begin{proposition}
    Let $A$ be an affine algebra and $M$ a simple module such that the cardinality of the base field $k$ is strictly greater than $\dim M$. Then $\End_A(M)$ is algebraic over $k$. \footnote{This lemma sometimes is wrongly attribued to Dixmier. In his version of the proposition, $A$ is countably generated and the field is uncountable}
\end{proposition}

\begin{proof}

    We offer the proof for it is quite easy. Let $\theta \in \End_A(M)$. If $\theta$ is not algebraic over $k$, then $M$ is a vector space over $k(\theta)$, which is a purely transcendental extension. Then $\dim_k \, M \geq \dim_k k(\theta) \geq |k|$, since in $k(\theta)$ the rational functions $(\theta-\lambda)^{-1}, \, \lambda \in k$ are linearly independent. However, this leads to a contradiction in cardinality. So $\theta$ is algebraic.
\end{proof}

Next we state a rather spectacular theorem of Quillen.

\begin{theorem}
    Let $A$ be a $k$-algebra with filtration $\mathcal{F}=\{A_i \}_{i \geq 0}$, not necessarily finite-dimensional and possibly with $k \subsetneq A_0$, such that $\gr_\mathcal{F} \, A$ is an affine commutative algebra. Then $A$ satisfies the endomorphism property.
\end{theorem}
\begin{proof}
    \cite{Quillen}.
\end{proof}

\begin{theorem}
Filtered semi-commutative algebras satisfy the endomorphism property.
\end{theorem}
\begin{proof}
    \cite{McConnnel2}.
\end{proof}

\section{More general filtrations and re-filtering}
\label{sec:more-general-filtrations}

In this last section we talk about algebras filtered by semigroups that are more general than $\mathbb{N}$, and we discuss re-filtering: when, from a infinite-dimensional filtration of an affine algebra, we extract a finite-dimensional filtration that has the same properties.

\begin{definition}
    A semigroup $G$ is called an \emph{ordered semigroup} if it has a total order $<$ such that, $x<y$ in $G$ implies $xz<yz$ abd $zx<zy$, for any $z \in G$.
\end{definition}

\begin{definition}
    Let $A$ be an affine algebra algebra and $G$ an ordered semigroup. We say that $A$ is \emph{filtered by $G$} if it has been equipped with a family $\mathcal{F}=\{ F_g|g \in G \}$ of subspaces of $A$ such that:
    \begin{enumerate}[{\rm (i)}]
        \item $F_g \subset F_h$ if $g<h$.
        \item $F_g F_h \subset F_{gh}$.
        \item Put $F_{<g}=\sum_{h \in G, h<g} F_h$. Then $A=\bigcup_{g \in G}(F_g \setminus F_{<g})$.
        \item $1 \in F_e \setminus F_{<e}$, where $e$ is the unity of $G$.
    \end{enumerate}
\end{definition}

Given a filtration as above, the \emph{associated graded algebra} $\gr_\mathcal{F} \, A$ is the vector space $\bigoplus_{g \in G} F_g \setminus F_{<g}$ with product $(a+F_{<g})(b+F_{<h})=(ab+F_{<gh})$.

\begin{example}
    Let $A$ be a $G$-graded algebra. Then we can introduce on it a $G$-filtration $\mathcal{F}=\{F_g|g \in G\}$ by $F_g=\bigoplus_{h \leq g} A_h$.
\end{example}

In this generality there is not much that one can say about the relation between the Gelfand-Kirillov dimension of $A$ and $\gr_\mathcal{F} \, A$. We limit ourselves to just one result.

\begin{theorem}
    Let $A$ be a domain filtered by an ordered semigroup $G$. Define $v:A \rightarrow \gr_\mathcal{F} \, A$ by $v(a)=a+F_{<g}$ for all $a \in F_g \setminus F_{<g}$. If $v(A)=\gr_\mathcal{F} \, A$, and $\gr_\mathcal{F} \, A$ is a domain, then $\GK \, A \geq \GK \, \gr_\mathcal{F} \, A$.
\end{theorem}

\begin{proof}
    \cite[Theorem 6.7]{Zhang}.
\end{proof}

\subsection{Multi-filtered algebras}

Through the rest of this section, the ordered monoid will be $\mathbb{N}^m$, $m>0$. In this case we say the algebras are multi-filtered. For a comprehensive discussion of the computational aspects of the Gelfand-Kirillov dimension of certain multi-filtered algebras related to quantum groups, see \cite{BOOK}.

\begin{definition}
    An admissible order $\preceq$ on $\mathbb{N}^m$ is a total ordering such that $(0,\ldots,0)$ is a minimal element in $\mathbb{N}^m$, and given $\alpha, \beta, \gamma \in \mathbb{N}^m$, $\alpha \prec \beta$ implies $\alpha+\gamma \preceq \beta + \gamma$. By Dickson's lemma \cite[Corollary 4.48]{Becker}, every admissible order is a well-order.
\end{definition}

\begin{definition}
    A \emph{multi-filtration} of an affine algebra $A$ is a collection $\{F_\gamma(A)|\gamma \in \mathbb{N}^m \}$ of subspaces of $A$ such that
    \begin{enumerate}[{\rm (i)}]
        \item $F_\gamma(A) \subset F_\delta(A)$ if $\gamma \preceq \delta$.
        \item $F_\gamma(A) F_\delta(A) \subset F_{\gamma+\delta}(A)$.
        \item $1 \in F_0(A)$
        \item $A=\bigcup_{\gamma \in \mathbb{N}^m} F_\gamma(A)$.
    \end{enumerate}

If all $F_\gamma(A)$ are finite-dimensional, we say that we have a finite-dimensional multi-filtration.    
\end{definition}

We already discussed how to construct the associated graded algebra, which we denote just by $G(A)$ for notational simplicity, in our more general construction with ordered semigroups. Now we will discuss modules.

\begin{definition}
    A \emph{multi-filtration} on a left module $M$ is a family $\{ F_\gamma(M)|\gamma \in \mathbb{N}^m\}$ of subspaces of $M$ such that
    \begin{enumerate}[{\rm (i)}]
        \item $F_\gamma(M) \subset F_\sigma(M)$ if $\gamma \preceq \sigma$,
        \item $F_\gamma(A)F_\sigma(M) \subset F_{\gamma+\sigma}(M)$,
        \item $M= \bigcup_{\gamma \in \mathbb{N}^m} F_\gamma(M)$.
    \end{enumerate}
If each $F_\gamma(M)$ is finite-dimensional, we say that the multi-filtration is \emph{finite-dimensional}.
\end{definition}

Call $F_{\prec \gamma}(M)=\bigcup_{\sigma \prec \gamma} F_\sigma(M)$.

For simplicity, we will denoted the associated graded module by $G(M)$. As a vector space,
\[ G(M) = \bigoplus_{\gamma \in \mathbb{N}^m} F_\gamma(M) \setminus F_{\prec \gamma}(M),\]
and we have multiplication $\big(a+F_{\prec \gamma}(A)\big)\big(m+F_{\prec \delta}(M)\big)=\big(am+ F_{\prec \gamma+\delta}(M)\big)$, turning it into a $\gr_{\mathbb{N}^m} \, A$-module.

We have the following generalization of a well known fact about $\mathbb{N}$-filtrations:

\begin{theorem}
    Let $A$ be a multi-filtered algebra. If the associated graded algebra is left or right Noetherian, so is $A$.
\end{theorem}
\begin{proof}
\cite[Theorem 1.5]{Torrecillas}.    
\end{proof}

The next result is the analogue of Proposition \ref{Proposition 6.6} for multi-filtrations.

\begin{proposition}
    Let $M$ be finitely generated multi-filtered module over an affine multi-filtered algebra $A$, and let $G(A)$ and $G(M)$ denote the associated graded algebra and module respectively. If $G(M)$ is a finitely generated $G(A)$-module, and $G(A)$ is affine, then
    \[ \GK_A \, M \geq \GK_{G(A)} \, G(M).\]
Furthermore, if the multi-filtrations are finite-dimensional, equality holds.
\end{proposition}
\begin{proof}
    \cite[Theorem 2.8]{Torrecillas}.
\end{proof}

Finally, for multi-filtrations we have an analogue of Tauvel's theorem for the exactness of the Gelfand-Kirillov dimension for finitely generated modules.

\begin{theorem}
    If $A$ is an affine algebra with a finite-dimensional multi-filtration with $F_0(A)=k$. If $G(A)$ is a left Noetherian affine algebra, then $\GK$ is exact on $\Amod$.
\end{theorem}
\begin{proof}
    \cite[Theorem 2.10]{Torrecillas}.
\end{proof}

\begin{example}
    If $A$ is the multiparameter quantized Weyl algebra $A_n^{Q,\Gamma}$ of Maltsiniotis \cite{Maltsiniotis} or $U_q(\mathfrak{sl}_n)$, $q \neq 0$ and not a root of unity, then $\GK$ is exact on $\Amod$.
    \cite{Torrecillas}.
\end{example}

\subsection{Re-filtering}

The main result of \cite{BTL} is that if we have an affine algebra $A$ with a multi-filtration, possibly infinite-dimensional, such that the associated graded algebra is semi-commutative, then it has a finite-dimensional $\mathbb{N}$-filtration such that the associated graded algebra is semi-commutative. This is called \emph{re-filtering}. This was not a new concept; see the already discussed paper \cite{{MStafford}}.

\begin{theorem}
    Let $A$ be an affine algebra. Suppose it has a multi-filtration such that that its associated graded algebra is semi-commutative. Then $A$ has a $\mathbb{N}$-filtration $\mathcal{F}$ which such that with respect to it $A$ is a filtered semi-commutative algebra.
\end{theorem}
\begin{proof}
    \cite[Theorem 2.3]{BTL}.
\end{proof}

Following \cite{BTL}, we call the algebras satisfying the hypothesis of the above theorem \emph{generalized semi-commutative}.

\begin{corollary}
    If $A$ is a generalized semi-commutative algebra, then it is a very nice algebra and it satisfies the endormorphism property.
\end{corollary}

\begin{example}
    The multiparamter quantized Weyl algebra $A_n^{Q, \Gamma}$ and $U_q(\mathfrak{g}), \, U_q(\mathfrak{g})^+$ are generalized semi-commutative \cite[Example 1.18]{BTL}.
\end{example}

\begin{remark}
    The ring theoretical properties for the associated graded algebra can be destroyed by the re-filtering process. With the usual multi-filtration, $G\big(U_q(\mathfrak{g})\big)$ is a domain; this is lost in the re-filtering.
\end{remark}

\end{document}